\documentclass[11pt,reqno]{amsart}

\newtheorem{proposition}{Proposition}[section]

\newtheorem{corollary}[proposition]{Corollary}
\newtheorem{theorem}[proposition]{Theorem}

\theoremstyle{definition}
\newtheorem{definition}[proposition]{Definition}
\newtheorem{example}[proposition]{Example}
\newtheorem{examples}[proposition]{Examples}
\newtheorem{remark}[proposition]{Remark}

\newcommand{\thlabel}[1]{\label{th:#1}}
\newcommand{\thref}[1]{Theorem~\ref{th:#1}}
\newcommand{\selabel}[1]{\label{se:#1}}

\newcommand{\prlabel}[1]{\label{pr:#1}}
\newcommand{\prref}[1]{Proposition~\ref{pr:#1}}
\newcommand{\colabel}[1]{\label{co:#1}}
\newcommand{\coref}[1]{Corollary~\ref{co:#1}}
\newcommand{\relabel}[1]{\label{re:#1}}
\newcommand{\reref}[1]{Remark~\ref{re:#1}}
\newcommand{\exlabel}[1]{\label{ex:#1}}

\newcommand{\delabel}[1]{\label{de:#1}}
\newcommand{\deref}[1]{Definition~\ref{de:#1}}
\newcommand{\eqlabel}[1]{\label{eq:#1}}
\newcommand{\equref}[1]{(\ref{eq:#1})}

\newcommand{\Aut}{{\rm Aut}\,}

\def\ZZ{{\mathbb Z}}

\newcommand{\Cc}{\mathcal{C}}
\newcommand{\Dd}{\mathcal{D}}

\def\*C{{}^*\hspace*{-1pt}{\Cc}}
\def\text#1{{\rm {\rm #1}}}

\input xy
\xyoption {all} \CompileMatrices

\usepackage{amssymb}
\usepackage{color,amssymb,graphicx,amscd}

\begin{document}

\title[Extending structures I: the level of groups]
{Extending structures I: the level of groups}

\author{A. L. Agore}
\thanks{A.L. Agore is research fellow ''Aspirant'' of FWO-Vlaanderen.
This work was supported by a grant of the Romanian National
Authority for Scientific Research, CNCS-UEFISCDI, grant no.
88/05.10.2011.}
\address{Faculty of Engineering, Vrije Universiteit Brussel, Pleinlaan 2, B-1050 Brussels, Belgium}
\email{ana.agore@vub.ac.be and ana.agore@gmail.com}
\author{G. Militaru}
\address{Faculty of Mathematics and Computer Science, University of Bucharest, Str.
Academiei 14, RO-010014 Bucharest 1, Romania}
\email{gigel.militaru@fmi.unibuc.ro and gigel.militaru@gmail.com}


%\maketitle

\begin{abstract} Let $H$ be a group and $E$ a set such that $H
\subseteq E$. We shall describe and classify up to an isomorphism
of groups that stabilizes $H$ the set of all group structures that
can be defined on $E$ such that $H$ is a subgroup of $E$. A
general product, which we call the unified product, is constructed
such that both the crossed product and the bicrossed product of
two groups are special cases of it. It is associated to $H$ and to
a system $\bigl( (S, 1_S, \ast), \triangleleft, \, \triangleright,
\, f \bigl)$ called a group extending structure and we denote it
by $H \ltimes S$. There exists a group structure on $E$ containing
$H$ as a subgroup if and only if there exists an isomorphism of
groups $(E, \cdot) \cong H \ltimes S$, for some group extending
structure $\bigl( (S, 1_S, \ast), \triangleleft, \,
\triangleright, \, f \bigl)$. All such group structures on $E$ are
classified up to an isomorphism of groups that stabilizes $H$ by a
cohomological type set ${\mathcal K}^{2}_{\ltimes} (H, (S, 1_S))$.
A Schreier type theorem is proved and an explicit example is
given: it classifies up to an isomorphism that stabilizes $H$ all
groups that contain $H$ as a subgroup of index 2.
\end{abstract}

\maketitle

\section*{Introduction}
The present paper is the starting point for a study concerning
what we have called the \emph{extending structures problem} or the
ES-problem for short. The ES-problem can be formulated at the most
general level using category theory language \cite{am-2011}. At
the level of groups the ES-problem has a very tempting statement:

\textbf{(Gr) Extending structures problem.} \textit{Let $H$ be a
group and $E$ a set such that $H \subseteq E$. Describe and
classify up to an isomorphism that stabilizes $H$ the set of all
group structures $\cdot$ that can be defined on $E$ such that $H$
is a subgroup of $(E, \cdot)$.}

In other words, the ES-problem is trying to provide an answer to
the very natural question: to what extent a group structure on $H$
can be extended beyond $H$ to a bigger set which contains $H$ as a
subset in such a way that $H$ would become a subgroup within the
new structure. The ES-problem generalizes and unifies two famous
problems in the theory of groups which served as models for our
approach: the \emph{extension problem} of H\"{o}lder \cite{holder}
and the \emph{factorization problem} of Ore \cite{Ore}. Let us
explain this briefly. Consider two groups $H$ and $G$. The
extension problem of H\"{o}lder consists of describing and
classifying all groups $E$ containing $H$ as a normal subgroup
such that $E/H \cong G$. An important step related to the
extension problem was made by Schreier: any extension $E$ of $H$
by $G$ is equivalent to a crossed product extension. For more
details and references on the extension problem we refer to the
monograph \cite{adem}.

The factorization problem is a "dual" of the extension problem and
it was formulated by Ore \cite{Ore}. It consists of describing and
classifying up to an isomorphism all groups $E$ that factorize
through $H$ and $G$: i.e. $E$ contains $H$ and $G$ as subgroups
such that $E = H G$ and $ H \cap G = \{1\}$. The dual version of
Schreier's theorem was proven by Takeuchi \cite{Takeuchi}: the
bicrossed product associated to a matched pair of groups $(H, G,
\triangleleft, \triangleright)$ was constructed and it was proven that a group $E$
factorizes through $H$ and $G$ if and only if $E$ is isomorphic to
a bicrossed product $H \bowtie \, G $. The factorization problem
is even more difficult than the more popular extension problem and
little progress has been made since then. For instance, in the
case of two cyclic groups $H$ and $G$, not both finite, the
problem was started by L. R\'{e}dei in \cite{Redei} and finished
by P.M. Cohn in \cite{Cohn}. If $H$ and $G$ are both finite cyclic
groups the problem seems to be still open, even though J. Douglas
\cite{Douglas} has devoted four papers to the subject. The case of
two cyclic groups, one of them being of prime order, was solved in
\cite{ACIM}. The bicrossed product, also known as knit product or
Zappa-Szep product in the theory of groups, appeared for the first
time in a paper by Zappa \cite{zappa} and it was rediscovered
later on by Szep \cite{Szep}.

In the construction of a crossed product a weak action $\alpha : G
\rightarrow \Aut (H)$ and an $\alpha$-cocycle $f : G \times G
\rightarrow H$ are used, while the construction of a bicrossed
product involves two compatible actions $\triangleright: G \times H
\rightarrow H$ and $\triangleleft : G \times H \rightarrow G$. Even if
their starting points are different, the two constructions have
something in common: the crossed product structure as well as the
bicrossed product structure are defined on the same set, namely $H
\times G$ and $H \cong (H, 1)$ is a subgroup in both the crossed
as well as the bicrossed product. Furthermore, $H \cong (H, 1)$ is
a normal subgroup in the crossed product. Conversely, any group
structure on a set $E$ containing $H$ as a normal subgroup can be
reconstructed from $H$ and the quotient $E/H$ as a crossed product
(see \reref{recon} for details). Now, if we drop the normality
assumption on $H$ the construction can not be performed anymore
and we have to come up with a new method of reconstructing a group
$E$ from a given subgroup and another set of data. This is what we
do in \thref{splitrec}, which is the first important result of the
paper. Let $H \leq E$ be a subgroup of a group $E$. Using the
axiom of choice, we can pick a retraction $p : E \to H$ of the
canonical inclusion $i : H \hookrightarrow E$ which is left
$H$-'linear'. Having this application we consider the pointed set
$S := p^{-1} (1)$, the fiber of $p$ in $1$. The group $H$ and the
pointed set $S$ are connected by four maps arising from $p$: two
actions $ \triangleright = \triangleright_p: S\times H \to H$, $\triangleleft = \triangleleft_p :
S\times H \to S$, a cocycle $f = f_p : S\times S \to H$ and a
multiplication $\ast = \ast_p : S\times S \to S$ constructed in
\thref{splitrec}. Using these maps, we shall prove that there
exists an isomorphism of groups $E \cong (H\times S, \cdot)$,
where the multiplication $\cdot$ on the set $H\times S$ is given
by
\begin{equation}\eqlabel{inmultireau00}
(h_{1}, s_{1}) \cdot (h_{2}, s_{2}) := \Bigl( \, h_{1} (s_{1}
\triangleright h_{2}) f(s_{1} \triangleleft h_{2}, s_{2}), \,\,
(s_{1} \triangleleft h_{2}) \ast s_{2} \Bigl )
\end{equation}
for all $h_1$, $h_2 \in H$ and $s_1$, $s_2\in S$. In other words,
even if we drop the normality assumption, the group $E$ can still
be rebuilt from a subgroup $H$ and the fiber $S$ of an $H$-linear
retraction. Moreover, any group structure $\cdot$ that can be
defined on a set $E$ such that a given group $H$ will be contained
as a subgroup has the form \equref{inmultireau00} for some system
$(S, \triangleright, \triangleleft, f, \ast )$. This new type of product will be
called \emph{unified product} and it is easily seen that both the
crossed and the bicrossed product of groups are special cases of
it. In the next step we will perform the abstract construction of
the unified product $H \ltimes S$: it is associated to a group $H$
and a system of data $\Omega (H) = \bigl( (S, 1_S, \ast),
\triangleleft, \, \triangleright, \, f \bigl)$ called extending
datum of $H$. \thref{constr} establishes the system of axioms that
has to be satisfied by $\Omega (H)$ such that $H\times S$ with the
multiplication defined by \equref{inmultireau00} becomes a group
structure, i.e. it is a unified product. In this case $\Omega (H)
= \bigl( (S, 1_S, \ast), \triangleleft, \, \triangleright, \, f
\bigl)$ will be called a \emph{group extending structure} of $H$.
Based on \thref{splitrec} and \thref{constr} we answer the
description part of the ES-problem in \coref{primaint}.
\thref{univers} shows the universality of the construction: the
unified product is at the same time an initial object in a certain
category and a final object in another category, which is not the
dual of the first one. The answer to the classification part of
the ES-problem is given in \thref{clasif1a}: the set of all group
structures $\cdot$ that can be defined on $E$ such that $H \leq
(E, \cdot)$ are classified up to an isomorphism of groups $\psi :
(E, \cdot) \to (E, \cdot')$ that stabilizes $H$ by a cohomological
type set ${\mathcal K}^{2}_{\ltimes} (H, (S, 1_S))$ which is
explicitly constructed. As a special case, a more restrictive
version of the classification is given in \coref{clasif2bb}, which
is a general Schreier theorem for unified products. This time all
unified products $H\ltimes S$ are classified up to an isomorphism
of groups that stabilizes both $H$ and $S$ by a set ${\mathcal
H}^{2}_{\ltimes} (H, (S, 1_S), \triangleleft)$ which plays for the
ES-problem problem the same role as the second cohomology group
from H\"{o}lder's extension problem. An explicit example is given in
\prref{indice2} where all groups that contain $H$ as a
subgroup of index $2$ are classified up to an isomorphism that
stabilizes $H$.

\section{Preliminaries}\selabel{prel}
Let $(S, 1_S)$ be a pointed set, i.e. $S$ is a non-empty set and
$1_S \in S$ is a fixed element in $S$. The group structures on a
set $H$ will be denoted using multiplicative notation and the unit
element will be denoted by $1_H$ or only $1$ when there is no
danger of confusion. $\Aut (H)$ denotes the group of automorphisms
of a group $H$ and $|S|$ the cardinal of a set $S$. A map $r: S
\to H$ is called unitary if $r(1_S) = 1_H$. $S$ is called a right
$H$-set if there exists a right action $\triangleleft : S \times H
\rightarrow S$ of $H$ on $S$, i.e.
\begin{equation}\eqlabel{ES1}
s \triangleleft (h_1 h_2) = (s\triangleleft {h_1} )\triangleleft
h_2 \quad \text{and} \quad s \triangleleft {1_H}= s
\end{equation}
for all $s\in S$, $h_1$, $h_2 \in H$. The action $\triangleleft : S \times
H \rightarrow S$ is called the trivial action if $s \triangleleft h = s$,
for all $s\in S$ and $h\in H$. Similarly the maps $ \triangleright : S
\times H \to H$ and $f: S \times S \to H$ are called trivial maps
if $ s \triangleright h = h$ and respectively $f (s_1, s_2) = 1_H$, for all
$s$, $s_1$, $s_2 \in S$ and $h\in H$. If $G$ is a group and
$\alpha : G \rightarrow \Aut (H)$ is a map we use the similar
notation $ \alpha (g) (h) = g\triangleright h$, for all $g\in G$
and $h\in H$.

\subsection*{Crossed product of groups}\selabel{1.1}
A \textit{crossed system} of groups is a quadruple $(H, G, \alpha,
f)$, where $H$ and $G$ are two groups, $\alpha : G \rightarrow
\Aut (H)$ and $f : G \times G \rightarrow H$ are two maps such
that the following compatibility conditions hold:
\begin{eqnarray}
g_1 \triangleright (g_2\triangleright h) &=& f(g_1, g_2) \,
\bigl((g_1g_2)\triangleright  h \bigl)\, f(g_1, g_2)^{-1}
\eqlabel{WA} \\
f(g_1,\, g_2)\, f(g_1 g_2, \, g_3) &=&  \bigl(g_1 \triangleright
f(g_2, \, g_3) \bigl) \, f(g_1, \, g_2g_3) \eqlabel{CC}
\end{eqnarray}
for all $g_1$, $g_2$, $g_3 \in G$ and $h\in H$. The crossed system
$\Gamma = (H, G, \alpha, f)$ is called \textit{normalized} if
$f(1, 1) = 1$. The map $\alpha : G \rightarrow \Aut (H)$ is called
a \textit{weak action} and $f : G \times G \rightarrow H$ is
called an $\alpha$-\textit{cocycle}. Let $H \#_{\alpha}^{f} \, G :
= H\times G$ as a set with a binary operation defined by the
formula:
\begin{equation}\eqlabel{41}
(h_1,\, g_1)\cdot (h_2, \, g_2) : = \bigl( h_1 (g_1 \triangleright
h_2) f(g_1, \, g_2), \, g_1g_2 \bigl)
\end{equation}
for all $h_1$, $h_2 \in H$, $g_1$, $g_2 \in G$. It is well known
that the multiplication on $H \#_{\alpha}^{f} \, G$ given by
\equref{41} is associative if and only if $(H, G, \alpha, f)$ is a
crossed system (see for instance \cite[Theorem 2.3]{am1}). In this
case $\bigl ( H \#_{\alpha}^{f} \, G, \, \cdot \bigl)$ is a group
with unit $1_{H \#_{\alpha}^{f} \, G} = \bigl( f(1, 1) ^{-1},
\, 1\bigl)$ called the \textit{crossed product of $H$ and $G$}
associated to the crossed system $(H, G, \alpha, f)$. A crossed
product with $f$ the trivial cocycle (that is $f(g_1, \, g_2) =
1_H$, for all $g_1$, $g_2 \in G$) is just $H \ltimes_{\alpha} G$,
the semidirect product of $H$ and $G$.

\begin{remark}\relabel{recon}
Schreier's theorem states that any extension $E$ of $H$ by a
group $G$ is equivalent to a crossed product extension. We recall
in detail this construction for further use. Let $H\unlhd E$ be a
normal subgroup of a group $E$ and $G := E/H$ the quotient group.
Let $\pi : E \to G$ be the canonical projection and $\chi : G \to
E$ be a section as a map of $\pi$ with $\chi (1_G) = 1_E$. Using
$\chi$ we define a weak action $\alpha$ and a cocycle $f$ by the
formulas:
\begin{equation}\eqlabel{1act}
\alpha : G \rightarrow \Aut (H), \qquad \, \alpha (g) (h) = g \triangleright
h : = \chi (g) h \chi (g)^{-1}
\end{equation}
\begin{equation}\eqlabel{1coc}
f : G \times G \rightarrow H, \qquad f (g_1, g_2) := \chi (g_1)
\chi (g_2) \chi (g_1 g_2)^{-1}
\end{equation}
for all $g\in G$, $g_1$, $g_2 \in G$ and $h\in H$. Then $(H, G,
\alpha, f)$ is a normalized crossed system of groups and
$$
\theta : H \#_{\alpha}^{f} \, G \to E, \qquad \theta (h, g) := h
\chi (g)
$$
is an isomorphism of groups. For more details we refer the reader
to \cite[Theorem 2.6]{am1}.
\end{remark}

\subsection*{Bicrossed product of groups}\selabel{1.2}
A \textit{matched pair} of groups is a quadruple $ (H, G, \triangleright,
\triangleleft) $, where $\triangleright: G \times H \rightarrow H$ is a left action
of the group $G$ on the set $H$, $\triangleleft : G \times H \rightarrow G$
is a right action of the group $H$ on the set $G$ such that:
\begin{equation}\eqlabel{2}
g \triangleright (h_{1} h_{2}) = (g \triangleright h_{1})\bigl((g \triangleleft h_{1}) \triangleright
h_{2}\bigl)
\end{equation}
\begin{equation}\eqlabel{3}
(g_{1} g_{2}) \triangleleft h = \bigl(g_{1} \triangleleft (g_{2} \triangleright h)\bigl)(g_{2}
\triangleleft h)
\end{equation}
for all $h$, $h_{1}$, $h_{2} \in H$ and $g$, $g_{1}$, $g_{2} \in
G$ (see \cite{Takeuchi}). Let $H \bowtie \, G : = H\times G$ as a
set with a binary operation defined by the formula:
\begin{equation}
(h_{1}, g_{1}) \cdot (h_{2}, g_{2}) = \bigl( h_{1}(g_{1} \triangleright
h_{2}), \, (g_{1} \triangleleft h_{2})g_{2}\bigl)
\end{equation}
for all $h_{1}$, $h_{2} \in H$ and $g_{1}$, $g_{2} \in G$. It can
be easily shown that $H \bowtie \, G $ is a group with $(1_{H},
1_{G})$ as a unit if and only if $(H, G, \triangleright, \triangleleft)$ is a matched
pair of groups. In this case $H \bowtie \, G $ is called the
\textit{bicrossed product, doublecross product, knit product or
Zappa-Szep product} associated to the matched pair $(H, G, \triangleright,
\triangleleft)$. Takeuchi proved \cite{Takeuchi} that a group $E$
factorizes through two subgroups $H$ and $G$ if and only if $E$ is
isomorphic to a bicrossed product $H \bowtie \, G$ associated to a
matched pair $(H, G, \triangleright, \triangleleft)$.

\section{Group extending structures and unified products}
The abstract definition of the unified product of groups will
arise from the following elementary question subsequent to the
ES-problem: \textit{let $H \leq E$ be a subgroup in $E$. Can we
reconstruct the group structure on $E$ from the one of $H$ and
some extra set of datum?} First we note that Schreier's classical
construction from \reref{recon} can not be used anymore. Thus we
should come up with a new method of reconstruction. The next
theorem indicates the way we can perform this reconstruction.

\begin{theorem}\thlabel{splitrec}
Let $H \leq E$ be a subgroup of a group $E$. Then:
\begin{enumerate}
\item There exists a map $p : E \to H$ such that $p (1) = 1$ and
for any $h\in H$, $x\in E$
\begin{equation}\eqlabel{33ab}
p ( h \, x ) = h \, p (x)
\end{equation}
\item For such a map $p : E \to H$ we define $S = S_p := p^{-1}(1)
= \{ x\in E \, |\, p (x) = 1 \}$. Then the multiplication map
\begin{equation}\eqlabel{335b}
\varphi : H \times S \to E, \quad \varphi ( h, s) := h s
\end{equation}
for all $h\in H$ and $s\in S$ is bijective with the inverse given
for any $x\in E$ by
$$
\varphi^{-1} : E \to H \times S, \quad \varphi^{-1} (x) = \bigl(
p(x), \, p (x)^{-1} \, x \bigl)
$$
\item For $p$ and $S$ as above there exist four maps $ \triangleright =
\triangleright_p: S\times H \to H$, $\triangleleft = \triangleleft_p : S\times H \to S$, $f =
f_p : S\times S \to H$ and $\ast = \ast_p : S\times S \to S$ given
by the formulas
\begin{eqnarray*}
s \triangleright h &{:=}& p (sh), \,\,\,\,\,\,\,\,\,\,\,  s\triangleleft h := p(sh)^{-1} s h \eqlabel{acti12}\\
f(s_1, s_2) &{:=}& p (s_1 s_2), \,\,\,\, s_1 \ast s_2 := p(s_1
s_2)^{-1} s_1 s_2
\end{eqnarray*}
for all $s$, $s_1$, $s_2 \in S$ and $h\in H$. Using these maps,
the unique group structure $'\cdot'$ on the set $H \times S$ such
that $\varphi : (H \times S, \cdot) \to E$ is an isomorphism of
groups is given by:
\begin{equation}\eqlabel{inmultireau}
(h_{1}, s_{1}) \cdot (h_{2}, s_{2}) := \Bigl( \, h_{1} (s_{1}
\triangleright h_{2}) f(s_{1} \triangleleft h_{2}, s_{2}), \,\,
(s_{1} \triangleleft h_{2}) \ast s_{2} \Bigl )
\end{equation}
for all $h_1$, $h_2 \in H$ and $s_1$, $s_2\in S$.
\end{enumerate}
\end{theorem}

\begin{proof}
(1) Using the axiom of choice we can fix $ \Gamma = (x_i)_{i\in I}
\subset E $ to be a system of representatives for the right
congruence modulo $H$ in $E$ such that $1 \in \Gamma$. Then for
any $x\in E$ there exists an unique $h_x \in H$ and an unique
$x_{i_0} \in \Gamma$ such that $x = h_x \, x_{i_0}$. Thus, there
exists a well defined map $p : E \to H$ given by the formula $p
(x) := h_x$, for all $x\in E$. As $1 \in \Gamma$ we have that $ p
(1) = 1$. Moreover, for any $h\in H$ and $x\in E$ we have that $h
x = h h_x \, x_{i_0}$. Thus $p (h x) = h h_x = h p(x)$, as needed.

(2) We note that $p (x)^{-1} x \in S$ as $ p \bigl( p (x)^{-1} \,
x \bigl) = p (x)^{-1} p(x) = 1 $, for all $x\in E$. The rest is
straightforward.

(3) First we note that $\triangleleft$ and $\ast$ are well defined maps.
Next, we can easily prove that the following two formulas hold:
\begin{eqnarray}
p(s_1 h_2 s_2) &=& (s_1 \triangleright h_2) f (s_1 \triangleleft h_2 , s_2)
\eqlabel{inmaa} \\
(s_1 \triangleleft h_2) \ast s_2 &=& p (s_1 h_2 s_2)^{-1} s_1 h_2 s_2
\eqlabel{inmab}
\end{eqnarray}
for all $s_1$, $s_2\in S$ and $h_2 \in H$. Indeed,
\begin{eqnarray*}
(s_1 \triangleright h_2) f (s_1 \triangleleft h_2 , s_2) &{=}& p(s_1 h_2) \, p\bigl(
(s_1
\triangleleft h_2) s_2 \bigl)\\
&{=}& p(s_1 h_2) \, p \bigl( p(s_1 h_2)^{-1} s_1 h_2 s_2  \bigl)\\
&\stackrel{\equref{33ab}} {=}& p (s_1 h_2 s_2)
\end{eqnarray*}
Similarly, we can prove that \equref{inmab} holds. Now, $\varphi :
H \times S \to E$ is a bijection between the set $H \times S$ and
the group $E$. Thus, there exists a unique group structure $\cdot$
on the set $H \times S$ such that $\varphi$ is an isomorphism of
groups. This group structure is obtained by transferring the group
structure from $E$ via the bijection $\varphi$, i.e. is given by:
\begin{eqnarray*}
(h_{1}, s_{1}) \cdot (h_{2}, s_{2}) &{=}& \varphi^{-1} \bigl(
\varphi (h_1, s_1) \varphi (h_2, s_2)  \bigl) \,\, = \varphi^{-1}
(h_1
s_1 h_2 s_2)\\
&{=}& \bigl( p(h_1 s_1 h_2 s_2), \, p(h_1 s_1 h_2 s_2)^{-1} h_1
s_1
h_2 s_2  \bigl)\\
&\stackrel{\equref{33ab}} {=}& \bigl( h_1 p(s_1 h_2 s_2), \, p(s_1
h_2 s_2)^{-1} s_1 h_2 s_2\bigl) \\
&\stackrel{\equref{inmaa}, \equref{inmab}} {=}& \Bigl( \, h_{1}
(s_{1} \triangleright h_{2}) f(s_{1} \triangleleft h_{2}, \,
s_{2}), \,\, (s_{1} \triangleleft h_{2}) \ast s_{2} \Bigl )
\end{eqnarray*}
for all $h_1$, $h_2 \in H$ and $s_1$, $s_2\in S$ as needed.
\end{proof}

\begin{remark}\relabel{cecon}
As $1\in S$ and $p(s) = 1$, for all $s\in S$ the maps $ \triangleright =
\triangleright_p$, $\triangleleft = \triangleleft_p$, $f = f_p $ and $\ast = \ast_p$
constructed in (3) of \thref{splitrec} satisfy the following
normalizing conditions:
\begin{equation}\eqlabel{nor1}
s \triangleright 1 = 1, \quad 1 \triangleright h = h, \quad 1
\triangleleft h = 1, \quad s\triangleleft 1 = s
\end{equation}
\begin{equation}\eqlabel{nor2}
f(s, 1) = f(1, s) = 1, \quad s\ast 1 = 1 \ast s = s
\end{equation}
for all $s\in S$ and $h\in H$. Hence, the multiplication $\ast$ on
$S$ has a unit but is not necessary associative. In fact, we can
easily prove that it satisfies the following compatibility:
\begin{equation}\eqlabel{ES7ab}
(s_{1} \ast s_{2}) \ast s_{3} = \bigl( s_{1} \triangleleft
f(s_{2},s_{3}) \bigl ) \ast (s_{2} \ast s_{3})
\end{equation}
for all $s_1$, $s_2$, $s_3 \in S$, i.e. $\ast$ is associative up
to the pair $(\triangleleft, f)$. Moreover, any element $s\in S$ is left
invertible in $(S, \ast)$; more precisely we can show that for any
$s\in S$ there exists a unique element $s' \in S$ such that $s'
\ast s = 1$.
\end{remark}

\subsection*{The abstract construction of the unified product}
Let $H$ be a group and $E$ a set such that $H \subseteq E$.
\thref{splitrec} describes the way any group structure $\cdot$ on
the set $E$ such that $H$ is a subgroup of $(E, \cdot)$ should
look like. We are left to find the abstract axioms that need to be fulfilled by
the system of maps $(\ast, \triangleleft, \triangleright, f)$ such that
\equref{inmultireau} is indeed a group structure. This will be
done below.

\begin{definition}\delabel{hgroup}
Let $H$ be a group. An \textit{extending datum of $H$} is a system
$\Omega(H) = \bigl( (S, 1_S, \ast), \triangleleft, \,
\triangleright, \, f \bigl)$ where:

$(1)$ $(S, 1_S)$ is a pointed set, $\ast: S \times S \to S$ is a
binary operation such that for any $s\in S$
\begin{equation}\eqlabel{ES3}
s\ast 1_S = 1_S \ast s = s
\end{equation}
$(2)$ The maps $\triangleleft : S\times H \to S$, $\triangleright: S\times H \to H$
and $f : S\times S \to H$ satisfy the following normalizing
conditions for any $s\in S$ and $h\in H$:
\begin{equation}\eqlabel{nor1a}
s\triangleleft 1_H = s, \,\, 1_S \triangleleft h = 1_S, \,\, 1_S
\triangleright h = h, \,\, s\triangleright 1_H = 1_H, \,\, f(s,
1_S) = f(1_S, s) = 1_H
\end{equation}
\end{definition}

Let $H$ be a group and $\Omega(H) = \bigl( (S, 1_S, \ast),
\triangleleft, \, \triangleright, \, f \bigl)$ an extending datum
of $H$. We denote by $H \ltimes_{\Omega(H)} S$ := $H \ltimes S$
the set $H\times S$ with the binary operation defined by the
formula:
\begin{equation}\eqlabel{13a}
(h_{1}, s_{1}) \cdot (h_{2}, s_{2}) := \Bigl( \, h_{1} (s_{1}
\triangleright h_{2}) f(s_{1} \triangleleft h_{2}, s_{2}), \,\,
(s_{1} \triangleleft h_{2}) \ast s_{2} \Bigl )
\end{equation}
for all $h_1$, $h_2 \in H$ and $s_1$, $s_2\in S$.

\begin{definition}\delabel{unifprod}
Let $H$ be a group and $\Omega(H) = \bigl( (S, 1_S, \ast),
\triangleleft, \, \triangleright, \, f \bigl)$ an extending datum
of $H$. The object $H \ltimes S$ introduced above is called
\textit{the unified product of $H$ and $\Omega(H)$} if $H \ltimes
S$ is a group with the multiplication given by \equref{13a}. In
this case the extending datum $\Omega(H)$ is called a
\textit{group extending structure} of $H$. The maps $\triangleright$ and
$\triangleleft$ are called the \textit{actions} of $\Omega(H)$ and $f$ is
called the $(\triangleright, \triangleleft)$-\textit{cocycle} of $\Omega(H)$.
\end{definition}

Using \equref{ES3} and \equref{nor1a} it is straightforward to
prove that $(1_H, 1_S)$ is a unit of the multiplication
\equref{13a} and the following relations hold in $H \ltimes S$:
\begin{eqnarray}
(h_1 ,  1_S)\cdot (h_2,  s_2) &{=}& (h_1 h_2,
s_2) \eqlabel{13}\\
(h_1,  s_1) \cdot (1_H,  s_2) &{=}& ( h_1 f(s_{1}, \, s_{2}),
s_{1}\ast s_{2})\eqlabel{14}\\
(h_1,  s_1) \cdot (h_2 , 1_{S}) &{=}& ( h_1 (s_{1} \triangleright
h_{2}) , s_{1} \triangleleft h_{2}) \eqlabel{15}
\end{eqnarray}
for all $h_1$, $h_2 \in H$ and $s_1$, $s_2\in S$. Next, we
indicate the abstract system of axioms that need to be satisfied
by the maps $ (\ast, \triangleleft, \, \triangleright, \, f)$
such that $H \ltimes S$ becomes a unified product.

\begin{theorem}\thlabel{constr}
Let $H$ be a group and $\Omega(H) = \bigl( (S, 1_S, \ast),
\triangleleft, \, \triangleright, \, f \bigl)$ an extending datum
of $H$. The following statements are equivalent:

$(1)$ $A \ltimes H$ is an unified product;

$(2)$ The following compatibilities hold for any $s$, $s_1$,
$s_2$, $s_3 \in S$ and $h$, $h_1$, $h_2 \in H$:

\begin{enumerate}
\item[(ES1)] The map $\triangleleft : S\times H \to S$ is a right action of
the group $H$ on the set $S$;

\item[(ES2)]  $(s_{1} \ast s_{2}) \ast s_{3} = \bigl( s_{1}
\triangleleft f(s_{2},s_{3}) \bigl ) \ast (s_{2} \ast s_{3});$

\item[(ES3)] $s \triangleright(h_{1} h_{2}) = (s \triangleright
h_{1}) \bigl( (s \triangleleft h_{1}) \triangleright h_{2}
\bigl);$

\item[(ES4)] $(s_{1} \ast s_{2}) \triangleleft h = \bigl ( s_{1}
\triangleleft (s_{2} \triangleright h) \bigl ) \ast (s_{2}
\triangleleft h);$

\item[(ES5)] $\bigl( s_{1} \triangleright(s_{2}\triangleright
h)\bigl) f \bigl( s_{1} \triangleleft (s_{2}\triangleright h), \,
s_{2} \triangleleft h \bigl ) =  f(s_{1}, \, s_{2}) \bigl( (s_{1}
\ast s_{2} ) \triangleright h \bigl);$

\item[(ES6)] $f(s_{1}, \, s_{2}) f(s_{1} \ast s_{2}, \, s_{3}) =
\bigl (s_{1} \triangleright f(s_{2}, \, s_{3})\bigl ) f\bigl
(s_{1} \triangleleft f(s_{2}, s_{3}), \, s_{2} \ast s_{3}\bigl );$

\item[(ES7)] For any $s\in S$ there exists $s'\in S$ such that $s'
\ast s = 1_S$.
\end{enumerate}
\end{theorem}

Before going into the proof of the theorem we note that $(ES3)$
and $(ES4)$ are exactly, mutatis-mutandis, the compatibility
conditions \equref{2} and \equref{3} from the definition of a
matched pair of groups while $(ES5)$ and $(ES6)$ are deformations
via the right action $\triangleleft$ of the compatibility conditions
\equref{WA} and \equref{CC} from the definition of a crossed
system of groups. The axiom $(ES1)$ is called the \emph{twisted
associativity condition} as it measures how far $\ast$ is from
being associative, i.e. from being a group structure on $S$.

\begin{proof}
We know that $(1_{H}, 1_{S})$ is a unit for the operation defined by
\equref{13a}. We prove now that $\cdot$ given by
\equref{13a} is associative if and only if the compatibility
conditions $(ES1) - (ES6)$ hold. Assume first that $\cdot$ is
associative and let $h$, $h_{1}$, $h_{2} \in H$ and $s$, $s_1$,
$s_{2} \in S$. The associativity condition
$$
[(1_H, s)\cdot (h_{1}, 1_S)]\cdot (h_2, 1_S) = (1_H, s)\cdot
[(h_{1}, 1_S)\cdot (h_2, 1_S)]
$$
gives, after we use the cross relations \equref{15} and
\equref{13}, $ (s \triangleright h_1, \, s\triangleleft h_1 ) \cdot (h_2, 1_S) =
(1_H, s) \cdot (h_1 h_2, \, 1_S)$. Thus $ \Bigl((s \triangleright
h_{1}) \bigl( (s \triangleleft h_{1}) \triangleright h_{2} \bigl),
\, (s\triangleleft h_1) \triangleleft h_2 \Bigl) = \bigl(s \triangleright (h_1 h_2), \, s \triangleleft
(h_1 h_2) \bigl)$ and hence $(ES1)$ and $(ES3)$ hold. Now, by
writing the associativity condition $[(1_H, s_1)\cdot (1_H,
s_2)]\cdot (1_H, s_3) = (1_H, s_1)\cdot [(1_H, s_2)\cdot (1_H,
s_3)]$ and computing this equality using the cross relation
\equref{14} it follows that the compatibility conditions
$(ES2)$ and $(ES6)$ hold. Finally, if we write the associativity
condition $[(1_H, s_1)\cdot (1_H, s_2)]\cdot (h, 1_S) = (1_H,
s_1)\cdot [(1_H, s_2)\cdot (h, 1_S)]$ and use \equref{14} and then
\equref{15} we obtain that $(ES4)$ and $(ES5)$
hold.

Conversely, assume that the compatibility conditions $(ES1) -
(ES6)$ hold. Then, by a rather long but straightforward
computation, which can be provided upon
request, we can prove that the operation $\cdot$ is associative,
that is $(h_{1}, s_{1})\cdot [(h_{2}, s_{2}) \cdot (h_{3}, s_{3})]
= [(h_{1}, s_{1})\cdot (h_{2}, s_{2})]\cdot (h_{3}, s_{3})$, for
all $h_1$, $h_2$, $h_{3} \in H$ and $s_1$, $s_2$, $s_{3} \in S$.

To conclude, we have proved that $( H \ltimes S , \cdot )$ is a monoid
if and only if $(ES1) - (ES6)$ hold. Assume now that $( H \ltimes
S , \cdot )$ is a monoid: it remains to be proved that the monoid
is actually a group if and only if $(ES7)$ holds. Indeed, in the
monoid $( H \ltimes S , \cdot )$ we have:
$$
(h, 1_S)\cdot (1_H, s) = (h, s), \qquad (h_1, 1_S) \cdot (h_2,
1_S) = (h_1 h_2, 1_S)
$$
for all $h$, $h_1$, $h_2 \in H$ and $s\in S$. In particular, any
element of the form $(h, 1_S)$, for $h\in H$ is invertible in $( H
\ltimes S , \cdot )$. Now, a monoid is a group if and only if each
of his elements has a left inverse. As $\cdot$ is associative it
follows from:
$$
(h^{-1}, 1_S)\cdot (h, s) = (1_H, s)
$$
that $( H \ltimes S , \cdot )$ is a group if and only if $(1_H,
s)$ has a left inverse for all $s\in S$. Hence, for any $s\in S$
there exist elements $s' \in S$ and $h'\in H$ such that
$$
(h' , s' ) \cdot (1_H, s) = \bigl ( h' f (s', s), \, s' \ast
s\bigl ) = (1_H, 1_S)
$$
This is of course equivalent to the fact that $s' \ast s = 1_S$
for all $s \in S$ and $h' = f\bigl( s', \, s \bigl)^{-1}$. The
proof is now finished. We note that the inverse of an element $(h,
s)$ in the group $( H \ltimes S , \cdot )$ is given by the formula
$$
(h, \, s)^{-1} = \bigl( f(s', \, s)^{-1} (s' \triangleright
h^{-1}), \, s' \triangleleft h^{-1}\bigl)
$$
where $s' \ast s = 1_S$.
\end{proof}

\begin{remark}\relabel{cumsefac}
Let $H$ be a group, $\Omega(H) = \bigl( (S, 1_S, \ast),
\triangleleft, \, \triangleright, \, f \bigl)$ a group extending
datum of $H$ and $H \ltimes_{\Omega} S$ the associated unified
product. Then the canonical inclusion
$$
i_H : H \to H \ltimes_{\Omega} S, \qquad i_H (h) := (h, 1_S)
$$
is a morphism of groups and the map
$$
p_H : H \ltimes_{\Omega} S \to H, \qquad p_H (h, s) := h
$$
satisfies condition \equref{33ab} of \thref{splitrec}.
Moreover, if we identify $H \cong (H, 1_S) \leq H \ltimes_{\Omega}
S$ and $S \cong (1_H, S) \subset H \ltimes_{\Omega} S$ we can
easily show that the maps $\ast$, $\triangleleft$, $\triangleright$ and $f$ from the
definition of $\Omega(H)$ are exactly the ones given in (3) of
\thref{splitrec} associated to the splitting map $p_H$.
\end{remark}

We record these observations in
the following result which gives the answer to the description
part of the ES-problem.

\begin{corollary}\colabel{primaint}
Let $H$ be a group and $E$ a set such that $H \subseteq E$. Then
there exists a group structure $\cdot$ on $E$ such that $H$ is a
subgroup of $(E, \cdot)$ if and only if there exists a group
extending structure $\Omega(H) = \bigl( (S, 1_S, \ast),
\triangleleft, \, \triangleright, \, f \bigl)$ of $H$ such that $
H \ltimes S \cong (E, \cdot)$.
\end{corollary}

\begin{proof}
It follows from \reref{cumsefac}, \thref{splitrec} and
\thref{constr}.
\end{proof}

In what follows we provide some special cases of unified products.
First of all we show that the unified product unifies both the
crossed product as well as the bicrossed product of groups.

\begin{example}\exlabel{cazpar}
Let $\Omega(H) = \bigl( (S, 1_S, \ast), \triangleleft, \,
\triangleright, \, f \bigl)$ be an extending datum of $H$ such
that $\triangleleft$ is the trivial action, that is $s \triangleleft h := s$, for
all $s \in S$ and $h\in H$. Then $\Omega (H) $ is a group
extending structure of $H$ if and only if $(S, \ast)$ is a group
structure on the set $S$ and $(H, (S, \ast), \triangleright, f)$ is a
crossed system of groups. In this case, the associated unified
product $H\ltimes_{\Omega} S = H \#_{\triangleright}^{f} \, G $ is the
crossed product of groups.
\end{example}

\begin{example}\exlabel{cazparb}
Let $\Omega(H) = \bigl( (S, 1_S, \ast), \triangleleft, \,
\triangleright, \, f \bigl)$ be an extending datum of $H$ such
that $f$ is the trivial cocycle, that is $f(s_1, s_2) = 1$, for
all $s_1$, $s_2\in S$. Then $\Omega (H) $ is a group extending
structure of $H$ if and only if $(S, \ast)$ is a group structure
on the set $S$ and $(H, (S, \ast), \triangleright, \triangleleft)$ is a matched pair
of groups. In this case, the associated unified product
$A\ltimes_{\Omega} H = H \bowtie \, G $ is the bicrossed product
of groups.
\end{example}

\begin{example}
There exist examples of groups that cannot be written either as a
crossed product or as a bicrossed product of two groups of smaller
order. Such a group should be a simple group (otherwise it can be
written as a crossed product). The simple group of smallest order
that cannot be written as a bicrossed product is the alternating
group $A_6$ (\cite{WW}). The above results allows us to write
$A_6$, and in fact any other simple group which is not a bicrossed
product, as a unified product between one of its subgroups and an
extending structure. For instance, we can write
$$
A_6 \cong A_4 \ltimes_{\Omega} S
$$
for an extending structure $\Omega (A_4) = \bigl( (S, 1_S, \ast),
\triangleleft, \, \triangleright, \, f \bigl)$, where $S$ is a set
with 30 elements.
\end{example}

Finally, an example of a group extending structure was constructed
in \cite{am2} as follows:

\begin{example}
Let $H$ be a group, $(S, 1_S, \ast)$ a pointed set with a binary
operation $ \ast : S \times S \to S$ having $1_S$ as a unit. Let $
\triangleleft : S \times H \to S$ be a map such that $(S, \triangleleft)$ is a right
$H$-set and $\gamma : S \to H$ be a map with $\gamma (1_S) = 1_H$
such that the following compatibilities hold
\begin{eqnarray*}
(x \ast y) \ast z &=& \Bigl ( x \triangleleft \bigl( \gamma(y) \, \gamma
(z) \, \gamma (y\ast z) ^{-1} \bigl) \Bigl ) \, \ast \,  (y \ast
z) \eqlabel{be1}\\
(x \ast y) \triangleleft  g &=& \Bigl( x \triangleleft \bigl( \gamma (y) \, g \,
\gamma ( y \triangleleft g )^{-1} \bigl) \Bigl) \, \ast \, (y \triangleleft g)
\eqlabel{be3}
\end{eqnarray*}
for all $g\in H$, $x$, $y$, $z\in S$. Using the transition map
$\gamma$ we define a left action $\triangleright$ and a cocycle $f$ via:
$$
x \triangleright g := \gamma (x) \, g \,\gamma (x \triangleleft g) ^{-1}, \quad f (x,
\, y) := \gamma (x) \, \gamma (y) \, \gamma (x \ast y)^{-1}
$$
for all $x$, $y\in S$ and $g\in H$. Then we can prove that $\Omega
(H) = \bigl( (S, 1_S, \ast), \triangleleft, \, \triangleright, \,
f \bigl)$ is a group extending structure of $H$.
\end{example}

\subsection*{The universal properties of the unified product}
In this subsection we prove the universality of the unified
product. Let $H$ be a group and $\Omega (H) = \bigl( (S, 1_S,
\ast), \triangleleft, \, \triangleright, \, f \bigl)$ a group
extending structure of $H$. We associate to $\Omega (H) $ two
categories ${}_{\Omega (H)}{\mathcal C}$ and ${\mathcal D}_{\Omega
(H)}$ such that the unified product becomes an initial object in
the first category and a final object in the second category.
Define the category ${}_{\Omega (H)}{\mathcal C}$ as follows: the
objects of ${}_{\Omega (H)}{\mathcal C}$ are pairs $(G, (u, v))$,
where $G$ is a group, $u : H\rightarrow G$ is a morphism of groups
and $v : S\rightarrow G$ is a map such that:
\begin{eqnarray}
v(s_1) v(s_2) &=& u(f(s_1, s_2)) v(s_1\ast s_2)\eqlabel{2.31}\\
v(s) u(h)&=& u(s\triangleright h)v(s\triangleleft h)\eqlabel{2.32}
\end{eqnarray}
for all $s$, $s_1$,  $s_2 \in S$ and $h \in H$. The morphisms of
the category $f:(G_1, (u_1, v_1))\rightarrow (G_2, (u_2, v_2))$
are morphisms of groups $f : G_1\rightarrow G_2$ such that :
$f\circ u_1 = u_2$ and $f\circ v_1=v_2$.

Define the category ${\mathcal D}_{\Omega (H)}$ as follows: the
objects of ${\mathcal D}_{\Omega (H)}$ are pairs $(G, (u,v) )$,
where $G$ is a group, $u : G\rightarrow H$, $v : G\rightarrow S$
are maps such that:
\begin{eqnarray}
u(xy) &=& u(x)[v(x)\triangleright u(y)]f(v(x)\triangleleft u(y),
v(y))\eqlabel{2.33}\\
v(xy) &=& [v(x)\triangleleft u(y)]\ast v(y)\eqlabel{2.34}
\end{eqnarray}
for all $x, y \in G$ while the morphisms of this category $f :
(G_1, (u_1, v_1)) \rightarrow (G_2, (u_2, v_2))$ are morphisms of
groups $f : G_1\rightarrow G_2$ such that: $u_2\circ f = u_1$ and
$v_2\circ f=v_1$.

\begin{theorem}\thlabel{univers}
Let $H$ be a group and $\Omega (H) = \bigl( (S, 1_S, \ast),
\triangleleft, \, \triangleright, \, f \bigl)$ a group extending
structure of $H$. Then:

\begin{enumerate}
\item[(1)] $(H \ltimes_{\Omega} S, (i_H, i_S))$ is an initial
object of ${}_{\Omega (H)}{\mathcal C}$, where $i_H : H\rightarrow
H \ltimes_{\Omega} S$ and $i_S:S\rightarrow H \ltimes_{\Omega} S$
are the canonical inclusions;

\item[(2)] $(H \ltimes_{\Omega} S, (\pi_H, \pi_S))$ is a final
object of ${\mathcal D}_{\Omega (H)}$, where $\pi_H : H
\ltimes_{\Omega} S\rightarrow H $ and $\pi_S :H \ltimes_{\Omega}
S\rightarrow S$ are the canonical projections.
\end{enumerate}
\end{theorem}

\begin{proof}
$(1)$ It is easy to see that $(H \ltimes_{\Omega} S, (i_H, i_S))$
is an object in the category ${}_{\Omega (H)}{\mathcal C}$. Let
$(G, (u,v))$ be an object in ${}_{\Omega (H)}{\mathcal C}$. We
need to prove that there exists an unique morphism of groups $\psi
: H \ltimes_{\Omega} S \rightarrow G$ such that the following
diagram commutes:
$$
\xymatrix{ H\ar[r]^{i_{H}}\ar[dr]_{u} & H \ltimes_{\Omega}
S\ar[d]^{\psi} & S\ar[l]_{i_{S}}\ar[dl]^{v} \\ &G }
$$
Assume first that $\psi$ satisfies the above condition. We obtain:
$\psi ((h,s)) = \psi((h, 1_S) \cdot(1_H, s)) = \psi ((h, 1_S))
\psi((1_H, s)) = (\psi \circ i_H)(h) (\psi \circ i_S)(s)) =
u(h)v(s)$, for all $h \in H$, $s \in S$ and we proved that $\psi$
is uniquely determined by $u$ and $v$. The existence of $\psi$ can
be proved as follows: we define $\psi: H\ltimes_{\Omega} S
\rightarrow G$ by $\psi ((h, s)) := u(h)v(s)$, for all $h\in H$
and $s\in S$. The fact that $\psi$ is a morphism of groups is just a straightforward
computation and the commutativity of the diagram is obvious.

$(2)$ Similar to $(1)$.
\end{proof}

\subsection*{The classification of unified products}
In this subsection we provide the classification part of the
ES-problem. Using \coref{primaint}, the classification of all
group structures on $E$ that contain $H$ as a subgroup, reduces
to the classification of all unified products $H \ltimes_{\Omega}
S$, associated to all group extending structures $\Omega (H) =
\bigl( (S, 1_S, \ast), \triangleleft, \, \triangleright, \, f
\bigl)$, for a set $S$ such that $|H| |S| = |E|$.

From now on the group $H$ and the pointed set $(S, 1_S)$ will be
fixed. Let ${\mathcal G} {\mathcal E} {\mathcal S}$$(H, (S, 1_S))$
be the set of all quadruples $(\ast, \triangleleft, \,
\triangleright, \, f)$ such that $\bigl( (S, 1_S, \ast),
\triangleleft, \, \triangleright, \, f \bigl)$ is a group
extending structure of $H$.

\begin{definition}\delabel{diagdef}
Let $\Omega (H) = \bigl( (S, 1_S, \ast), \triangleleft, \,
\triangleright, \, f \bigl)$ and $\Omega ' (H) = \bigl( (S, 1_{S},
\ast'), \triangleleft', \, \triangleright', \, f' \bigl)$ be two
group extending structures of $H$ and $H \ltimes_{\Omega} S$, $H
\ltimes_{\Omega'} S$ the associated unified products. For a
morphism of groups $\psi : H \ltimes_{\Omega} S \to H
\ltimes_{\Omega'} S$ we consider the following diagram
\begin{equation}\eqlabel{D2}
\xymatrix {& H \ar[r]^{i_{H}} \ar[d]_{Id_{H}} & {H\ltimes S}
\ar[r]^{\pi_{S}}\ar[d]^{\psi} & S\ar[d]^{Id_{S}}\\
& H\ar[r]^{i_{H}} & {H\ltimes' S}\ar[r]^{\pi_{S}} & S}
\end{equation}
where  $\pi_S : H \ltimes_{\Omega} S \to S$ is the canonical
projection $\pi (h, s) := s$, for all $h\in H$ and $s\in S$. We
say that $\psi : H \ltimes_{\Omega} S \to H \ltimes_{\Omega'} S$
\emph{stabilizes} $H$ (resp. stabilizes $S$) if the left square
(resp. the right square) of the diagram \equref{D2} is
commutative.
\end{definition}

\begin{proposition}\prlabel{clasif1}
Let $H$ be a group, $(S, 1_S)$, $(S', 1_{S'})$ two pointed sets
and $\Omega (H) = \bigl( (S, 1_S, \ast), \triangleleft, \,
\triangleright, \, f \bigl)$ and $\Omega ' (H) = \bigl( (S',
1_{S'}, \ast'), \triangleleft', \, \triangleright', \, f' \bigl)$
be two group extending structures of $H$. Then there exists a
bijective correspondence between the set of all morphisms of
groups $\psi : H \ltimes_{\Omega} S \to H \ltimes_{\Omega'} S'$
that stabilize $H$ and the set of all pairs $(r, v)$, where $r: S
\rightarrow H$, $v: S \rightarrow S'$ are two unitary maps such
that:
\begin{eqnarray}
v(s \triangleleft h) &=& v(s) \triangleleft' h \eqlabel{p1} \\
(s \triangleright h) \, r(s \triangleleft h) &=& r(s)\, \bigl(v(s) \triangleright'
h \bigl)\eqlabel{p2} \\
v(s_{1} \ast s_{2}) &=& \bigl( v(s_{1}) \triangleleft' r(s_{2})\bigl) \,
\ast' \, v(s_{2}) \eqlabel{p3}\\
f(s_1, s_2) \, r(s_{1} * s_{2}) &=& r(s_{1}) \, \bigl( v(s_{1})
\triangleright' r(s_{2}) \bigl) \, f' \bigl( v(s_{1}) \triangleleft' r(s_{2}), \,
v(s_2) \bigl) \eqlabel{p4}
\end{eqnarray}
for all $s$, $s_{1}$, $s_{2} \in S$ and $h \in H$. Through the
above correspondence the morphism $\psi : H \ltimes_{\Omega} S \to
H \ltimes_{\Omega'} S'$ corresponding to $(r, v)$ is given by
\begin{equation}\eqlabel{p5}
\psi(h, \, s) = \bigl(h \, r(s), \, v(s)\bigl)
\end{equation}
for all $h \in H$, $s \in S$. Furthermore, $\psi : H
\ltimes_{\Omega} S \to H \ltimes_{\Omega'} S'$ given by
\equref{p5} is an isomorphisms of groups if and only if $v: S \to
S'$ is bijective.
\end{proposition}

\begin{proof}
A morphism of groups $\psi : H \ltimes_{\Omega} S \to H
\ltimes_{\Omega'} S'$ that makes the left square of the diagram
\equref{D2} commutative is uniquely determined by two maps $r =
r_{\psi}: S \rightarrow H$, $v = v_{\psi}: S \rightarrow S'$ such
that $\psi(1, s) = \bigl( r(s), v(s)\bigl)$ for all $s \in S$. In
this case $\psi$ is given by :
$$
\psi(h, s) = \psi\bigl((h, 1_S)\cdot (1_H, s)\bigl) = (h, 1_S)
\cdot \bigl(r(s), v(s)\bigl)  = \bigl(h r(s), v(s)\bigl)
$$
for all $h \in H$ and $s \in S$. Now, $\psi(1_H, 1_S) = (1_H, 1_{S'})$ if and only if $r$
and $v$ are unitary maps. Assuming this unitary condition, we can
easily prove that $\psi$ is a morphism of groups if and only if
the compatibility conditions \equref{p1} - \equref{p4} hold for
the pair $(r, v)$. It remains to be proved that $\psi$ given by
\equref{p5} is an isomorphism if and only if $v: S\to S'$ is a
bijective map. Assume first that $\psi$ is an isomorphism. Then
$v$ is surjective and for $s_{1}$, $s_{2} \in S$ such that
$v(s_{1}) = v(s_{2})$ we have:
$$
\psi(1_H, s_{2}) = \bigl(r(s_{2}), v(s_{2})\bigl) =
\bigl(r(s_{2}), v(s_{1})\bigl) = \psi\bigl( r(s_{2})
r(s_{1})^{-1}, \, s_{1}\bigl)
$$
Hence $s_{1} = s_{2}$ and $v$ is injective. Conversely is
straightforward.
\end{proof}

\begin{definition}\delabel{equiv1} Two elements
$(\ast, \triangleleft, \, \triangleright, \, f)$ and $(\ast',
\triangleleft', \, \triangleright', \, f')$ of ${\mathcal G}
{\mathcal E} {\mathcal S}$$(H, (S, 1_S))$ are called
\emph{equivalent} and we denote this by $(\ast, \triangleleft, \,
\triangleright, \, f) \sim (\ast', \triangleleft', \,
\triangleright', \, f')$ if there exists a pair $(r, v)$ of
unitary maps  $r: S \rightarrow H$, $v: S \rightarrow S$ such that
$v$ is a bijection on the set $S$ and the compatibility conditions
\equref{p1} - \equref{p4} are fulfilled.
\end{definition}

It follows from \prref{clasif1} that $(\ast, \triangleleft, \,
\triangleright, \, f) \sim (\ast', \triangleleft', \,
\triangleright', \, f')$ if and only if there exists $\psi : H
\ltimes_{\Omega} S \to H \ltimes_{\Omega'} S$ an isomorphism of
groups that stabilizes $H$. Thus, $\sim$ is an equivalence
relation on the set ${\mathcal G} {\mathcal E} {\mathcal S}$$(H,
(S, 1_S))$. We denote by ${\mathcal K}^{2}_{\ltimes} (H, (S,
1_S))$ the quotient set ${\mathcal G} {\mathcal E} {\mathcal
S}$$(H, (S, 1_S)) /\sim$.

Let $\Cc (H, (S, 1_S) )$ be the category whose class of objects is
the set ${\mathcal G} {\mathcal E} {\mathcal S}$$(H, (S, 1_S))$. A
morphism $\psi : (\ast, \triangleleft, \, \triangleright, \, f)
\to (\ast', \triangleleft', \, \triangleright', \, f')$ in $\Cc
(H, (S, 1_S) )$ is a morphism of groups $\psi : H \ltimes_{\Omega}
S \to H \ltimes_{\Omega'} S$ that stabilizes $H$.

The main result of this paper which
gives the full answer to the ES-problem now follows as a
direct application of \prref{clasif1}: the classifying object for
the ES-problem is ${\mathcal K}^{2}_{\ltimes} (H, (S, 1_S))$
constructed above.

\begin{theorem} \thlabel{clasif1a}
Let $H$ be a group and $(S, 1_S)$ a pointed set. Then there exists
a bijection between the set of objects of the skeleton of the
category $\Cc (H, (S, 1_S) )$ and ${\mathcal K}^{2}_{\ltimes} (H,
(S, 1_S))$.
\end{theorem}

\begin{remark}\relabel{pana}
At the level of algebras the bicrossed product is known under the
name of \emph{twisted tensor product algebra}. \cite[Theorem
4.4]{panaite} provides sufficient conditions for two bicrossed
products of algebras $A\bowtie B$ and $A ' \bowtie B$ to be
isomorphic such that the isomorphism stabilizes $B$. The result in
\cite[Theorem 4.4]{panaite} can be improved and generalized in the
spirit of \prref{clasif1}.
\end{remark}

Using \prref{clasif1} we can also prove a general Schreier
classification theorem for unified products.

\begin{proposition}\prlabel{clasif2}
Let $\Omega (H) = \bigl( (S, 1_S, \ast), \triangleleft, \,
\triangleright, \, f \bigl)$, $\Omega ' (H) = \bigl( (S, 1_{S},
\ast'), \triangleleft', \, \triangleright', \, f' \bigl)$ be two
group extending structures of a group $H$. Then there exists a
morphism $\psi : H \ltimes_{\Omega} S \to H \ltimes_{\Omega'} S$
that stabilizes $H$ and $S$ if and only if $\triangleleft = \triangleleft'$ and there
exists a unitary map $r: S \rightarrow H$ such that $\triangleright$, $\ast$
and $f$ are implemented by $\triangleright '$, $\ast'$ and $f'$ via $r$ as
follows:
\begin{eqnarray}
s \triangleright h &=& r(s)\, \bigl(s \triangleright'
h \bigl) \, r(s \triangleleft h)^{-1} \eqlabel{p2b} \\
s_{1} \ast s_{2} &=& \bigl( s_{1} \triangleleft r(s_{2})\bigl) \,
\ast' \, s_{2} \eqlabel{p3b}\\
f(s_1, s_2) &=& r(s_{1}) \, \bigl( s_{1} \triangleright' r(s_{2}) \bigl) \,
f' ( s_{1} \triangleleft r(s_{2}), \, s_2 ) \, r(s_{1} * s_{2})^{-1}
\eqlabel{p4b}
\end{eqnarray}
for all $s$, $s_{1}$, $s_{2} \in S$ and $h \in H$. Furthermore,
any morphism of groups $\psi : H \ltimes S \to H \ltimes' S$ that
stabilizes $H$ and $S$ is an isomorphism of groups and is given by
\begin{equation}\eqlabel{p5b}
\psi(h, \, s) = \bigl(h \, r(s), \, s\bigl)
\end{equation}
for all $h \in H$, $s \in S$.
\end{proposition}

\begin{proof}
Indeed, using \prref{clasif1} any morphism of groups $\psi : H
\ltimes S \to H \ltimes 'S$ that makes the left square of
\equref{D2} commutative is given by \equref{p5} for some unique
maps $(u, v)$. Now, such a morphism $\psi = \psi_{u, v}$ makes the
right square of \equref{D2} commutative if and only if $v$ is the
identity map on $S$. Now the proof follows from \prref{clasif1}:
\equref{p1} implies that the right actions $\triangleleft$ and $\triangleleft'$
should be equal while $\equref{p2b}- \equref{p4b}$ are exactly
$\equref{p2}- \equref{p4}$ for $v = {\rm Id}_S$ and $\triangleleft =
\triangleleft'$.
\end{proof}

\prref{clasif2} tells us that in order to obtain a Schreier type
theorem for unifed products we have to set the group $H$, the
pointed set $(S, 1_S)$ and a right $H$-action $\triangleleft$ of the group
$H$ on the set $S$. Let ${\mathcal S} {\mathcal E} {\mathcal
S}$$(H, (S, 1_S), \triangleleft)$ be the set of all triples $(\ast,
\triangleright, \, f)$ such that $\bigl( (S, 1_S, \ast),
\triangleleft, \, \triangleright, \, f \bigl)$ is a group
extending structure of $H$.

\begin{definition}\delabel{equiv111}
Let $H$ be a group, $(S, 1_S)$ a pointed set and $\triangleleft : S \times
H \to S$ a right action of $H$ on $S$. Two elements $(\ast, \,
\triangleright, \, f)$ and $(\ast', \, \triangleright', \, f')$ of
${\mathcal S} {\mathcal E} {\mathcal S}$$(H, (S, 1_S), \triangleleft)$ are
called \emph{cohomologous} and we denote this by $(\ast, \,
\triangleright, \, f) \approx (\ast', \, \triangleright', \, f')$
if there exists a unitary map  $r: S \rightarrow H$ such that
\begin{eqnarray*}
s_{1} \ast s_{2} &=& \bigl( s_{1} \triangleleft r(s_{2})\bigl) \,
\ast' \, s_{2} \eqlabel{p3b1}\\
s \triangleright h &=& r(s)\, \bigl(s \triangleright'
h \bigl) \, r(s \triangleleft h)^{-1} \eqlabel{p2b1} \\
f(s_1, s_2) &=& r(s_{1}) \, \bigl( s_{1} \triangleright' r(s_{2}) \bigl) \,
f' ( s_{1} \triangleleft r(s_{2}), \, s_2 ) \, r(s_{1} * s_{2})^{-1}
\eqlabel{p4b1}
\end{eqnarray*}
for all $s$, $s_{1}$, $s_{2} \in S$ and $h \in H$.
\end{definition}

It follows from \prref{clasif2} that $(\ast, \, \triangleright, \,
f) \approx (\ast', \, \triangleright', \, f')$ if and only if
there exists $\psi : H \ltimes_{\Omega} S \to H \ltimes_{\Omega'}
S$ a morphism of groups such that diagram \equref{D2} is
commutative and moreover such a morphism is an isomorphism. Thus,
$\approx$ is an equivalence relation on the set ${\mathcal S}
{\mathcal E} {\mathcal S}$$(H, (S, 1_S), \triangleleft)$. We denote by
${\mathcal H}^{2}_{\ltimes} (H, (S, 1_S), \triangleleft)$ the quotient set
${\mathcal S} {\mathcal E} {\mathcal S}$$(H, (S, 1_S), \triangleleft)
/\approx$.

Let $\Dd (H, (S, 1_S), \triangleleft )$ be the category whose class of
objects is the set ${\mathcal S} {\mathcal E} {\mathcal S}$$(H,
(S, 1_S), \triangleleft)$. A morphism $\psi : (\ast, \, \triangleright, \,
f) \to (\ast', \, \triangleright', \, f')$ in $\Dd (H, (S, 1_S),
\triangleleft )$ is a morphism of groups $\psi : H \ltimes_{\Omega} S \to H
\ltimes_{\Omega'} S$ such that diagram \equref{D2} is
commutative. The category $\Dd (H, (S, 1_S), \triangleleft )$ is a
groupoid, that is any morphism is an isomorphism. We obtain:

\begin{corollary}\textbf{(The Schreier theorem for unified
products)}\colabel{clasif2bb} Let $H$ be a group, $(S, 1_S)$ a
pointed set and $\triangleleft$ a right action of $H$ on $S$. Then there
exists a bijection between the set of objects of the skeleton of
the category $\Dd (H, (S, 1_S), \triangleleft )$ and ${\mathcal
H}^{2}_{\ltimes} (H, (S, 1_S), \triangleleft)$.
\end{corollary}

Thus it follows from \coref{clasif2bb} that ${\mathcal
H}^{2}_{\ltimes} (H, (S, 1_S), \triangleleft)$ is for the classification of
unified products of groups the counterpart of the second
cohomology group for the classification of extensions of an
abelian group by a group \cite[Theorem 7.34]{R}.

\begin{corollary}\colabel{cazspec1}
Let $H$ be a group, $\Omega (H) = \bigl( (S, 1_S, \ast),
\triangleleft, \, \triangleright, \, f \bigl)$ a group extending
structure of $H$ and $(H, G, \triangleright', f')$ a crossed system of
groups. Then there exists $\psi : H \ltimes_{\Omega} S \to
H\#_{\triangleright'}^{f'} G$ an isomorphism of groups that stabilizes $H$ if
and only if the right action $\triangleleft$ of $\Omega (H)$ is the trivial
one and there exists a pair $(r, v)$, where $r: S \rightarrow H$
is a unitary map, $v: (S, \ast) \rightarrow G$ is an isomorphism
of groups such that:
\begin{eqnarray}
s \triangleright h &=&  r(s)\, \bigl(v(s) \triangleright'
h \bigl) \, r(s)^{-1}  \eqlabel{p2xx} \\
f(s_1, s_2) &=& r(s_{1}) \, \bigl( v(s_{1}) \triangleright' r(s_{2}) \bigl)
\, f' \bigl( v(s_{1}), \, v(s_2) \bigl) \,\, r(s_{1} * s_{2})^{-1}
\eqlabel{p4xx}
\end{eqnarray}
for all $s$, $s_{1}$, $s_{2} \in S$ and $h \in H$.
\end{corollary}

\begin{proof}
We apply \prref{clasif1} in the case that $\triangleleft'$ is the trivial
action in the group extending structure $\Omega ' (H) = \bigl( (G,
1_{G}, \ast'), \triangleleft', \, \triangleright', \, f' \bigl)$.
\end{proof}

Now we give necessary and sufficient conditions for a unified
product to be isomorphic to a bicrossed product of groups such
that $H$ is stabilized.

\begin{corollary}\colabel{cazspec2}
Let $H$ be a group, $\Omega (H) = \bigl( (S, 1_S, \ast),
\triangleleft, \, \triangleright, \, f \bigl)$ a group extending
structure of $H$ and $(H, G, \triangleleft', \triangleright')$ a matched pair of
groups. Then there exists $\psi : H \ltimes_{\Omega} S \to H
\bowtie G$ an isomorphism of group that stabilizes $H$ if and only
if there exists a pair $(r, v)$, where $r: S \rightarrow H$ is a
unitary map, $v: (S, \triangleleft) \rightarrow (G, \triangleleft') $ is a unitary
map and an isomorphism of $H$-sets such that:
\begin{eqnarray}
(s \triangleright h) \, r(s \triangleleft h) &=& r(s)\, \bigl(v(s) \triangleright'
h \bigl)\eqlabel{p2yy} \\
v(s_{1} \ast s_{2}) &=& \bigl( v(s_{1}) \triangleleft' r(s_{2})\bigl) \,
 \, v(s_{2}) \eqlabel{p3yy}\\
f(s_1, s_2) &=& r(s_{1}) \, \bigl( v(s_{1}) \triangleright' r(s_{2}) \bigl)
\, \, r(s_{1} * s_{2})^{-1} \eqlabel{p4yy}
\end{eqnarray}
for all $s$, $s_{1}$, $s_{2} \in S$ and $h \in H$.
\end{corollary}

\begin{proof}
We apply \prref{clasif1} in the case when $f'$ is the trivial
cocycle in the group extending structure $\Omega ' (H) = \bigl(
(G, 1_{G}, \ast'), \triangleleft', \, \triangleright', \, f'
\bigl)$.
\end{proof}

\thref{clasif1a} offers the theoretical answer to the ES-problem
at the level of groups. The challenge that remains is a
computational one: for a given group $H$ and a pointed set $(S,
1_S)$ we have to compute explicitly the cohomological object
${\mathcal K}^{2}_{\ltimes} (H, (S, 1_S))$: it classifies up to an
isomorphism of groups that stabilizes $H$ all groups that contain
$H$ as a subgroup of index $|S|$. We end the paper with an
explicit example. For a group $H$ we shall denote by ${\mathcal T}
(H) \subseteq H \times {\rm Aut} (H)$ the set consisting of all
pairs $(h_0, \, D)$, where $h_0\in H$, $D : H \to H$ is an
automorphism of $H$ such that for any $h\in H$:
\begin{equation}\eqlabel{ind2aa}
D(h_0) = h_0, \qquad D^2 (h) = h_0 h h_0^{-1}
\end{equation}

\begin{proposition}\prlabel{indice2}
Let $H$ be a group and $S = \{1_S, \, c\}$ a set with two
elements. Then:

\begin{enumerate}
\item[(1)] There exists a bijection between the set ${\mathcal G}
{\mathcal E} {\mathcal S}$$(H, (S, 1_S))$ of all group extending
structures of $H$ and ${\mathcal T} (H)$. The bijection is given
such that the group extending structure $\bigl( (S, 1_S, \ast),
\triangleleft, \, \triangleright, \, f \bigl)$ corresponding to
$(h_0, \, D) \in {\mathcal T} (H)$ is defined as follows:
$\triangleleft$ is the trivial action of $H$ on $S$, $\ast$ is
given by $c \ast c = 1_S$, the left action $\triangleright : S
\times H \to H$ and the cocycle $f: S \times S \to H$ are given
for any $h\in H$ by:
\begin{equation}\eqlabel{ind2}
c \triangleright h := D(h), \qquad f(c, c) := h_0
\end{equation}

\item[(2)] There exists a bijection ${\mathcal K}^{2}_{\ltimes}
(H, (S, 1_S)) \cong {\mathcal T} (H)/ \sim $, where $\sim$ is the
equivalence relation defined as follows: $(h_0, D) \sim (h'_0,
D')$ if and only if there exists $g\in H$ such that for any $h\in
H$
\begin{equation}\eqlabel{ind2bb}
h_0 = g D' (g) \, h'_0, \qquad D(h) = g D'(h) g^{-1}
\end{equation}
The bijection between ${\mathcal T} (H)/ \sim $ and the
isomorphism classes of all groups that contain and stabilize $H$
as a subgroup of index $2$ is given by
$$
\overline{(h_0, \, D)} \mapsto H\ltimes_{(h_0, D)} S
$$
where we denote by $H\ltimes_{(h_0, D)} S$ the unified product
associated to the group extending structure constructed in
\equref{ind2} for a given $(h_0, D) \in {\mathcal T} (H)$ and
$\overline{(h_0, \, D)}$ is the equivalence class of $(h_0, \, D)$
via the relation $\sim$. Explicitly, the multiplication on the
group $H\ltimes_{(h_0, D)} S$ is given for any $h$, $h_{1}$, $h_2
\in H$ by:
\begin{eqnarray*}
(h_{1}, 1_{S}) \cdot (h_{2}, 1_{S}) &=& (h_{1} h_{2}, 1_{S}),
\qquad (1, c) \cdot (1, c) = (h_{0}, 1_{S}) \\
(h, 1_{S}) \cdot (1, c) &=& (h, c), \qquad \quad \,\, (1, c) \cdot
(h, 1_{S}) = (D(h), c)
\end{eqnarray*}
\end{enumerate}
\end{proposition}

\begin{proof}
$(1)$ We have to compute the set of all maps $(\ast,
\triangleleft, \, \triangleright, \, f)$ satisfying the
normalizing conditions \equref{ES3}, \equref{nor1a} as well as the
compatibility conditions $(ES1)$-$(ES7)$. First we prove that $c
\triangleleft h = c$, for all $h\in H$, i.e. $\triangleleft$ is
the trivial action and $c \ast c = 1_S$, i.e. $S = C_2$, the
cyclic group of order $2$. Indeed, using \equref{nor1a} we already
know that $1_S \triangleleft h = 1_S$, for all $h \in H$. Let $f :
H \to S$ be a map such that $ c \triangleleft h = f(h)$, for all
$h \in H$. Then, $f$ is a unit preserving map since \equref{nor1a}
holds. We will prove that $f (h) = c$, for all $h \in H$, and
hence $\triangleleft$ is the trivial action. Indeed, $(ES1)$ tells
us that $\triangleleft$ is a right action: the condition $c
\triangleleft (g h) = (c \triangleleft g) \triangleleft h$ takes
the form $f (g h) = f(g) \triangleleft h$, for all $g$, $h \in H$.
Assume, that there exists $g \in H$ such that $f (g) = 1_S$. Then
we obtain that $f (g h) = 1_S \triangleleft h = 1_S$, i.e. $f (gh)
= 1_S$, for all $h\in H$. Thus, $f$ is the trivial map, that is $c
\triangleleft h = 1_S$, for all $h \in H$. Then, axiom $(ES2)$
applied for $s_1 = s_2 = s_3 = c$ implies that $ (c \ast c) \ast c
= c \ast c$ and using $(ES7)$ we obtain $c = 1_S$, which is a
contradiction.

Thus, $\triangleleft$ is the trivial action; it follows from
$(ES1)$ that $\ast$ is a group structure on $S$ i.e. $c \ast c =
1_S$. Now, any normalizing map $ \triangleright : S \times H \to
H$ is uniquely implemented by a map $D : H \to H$ such that $c
\triangleright h = D(h)$, for all $h\in H$, and any normalized map
$f : S \times S \to H$ is uniquely determined by an element $h_0
\in H$ such that $f (c, c) = h_0$. Now we can
easily see that the compatibility conditions $(ES1)$, $(ES2)$,
$(ES4)$ are trivially fulfilled, while $(ES3)$ is equivalent to
the fact that $D$ is an endomorphism of $H$, axiom $(ES5)$
takes the equivalent form given by the right hand side of
\equref{ind2aa} while $(ES6)$ is the left hand side of
\equref{ind2aa}.

$(2)$ Let $(\ast, \triangleleft, \, \triangleright, \, f)$ and
$(\ast', \triangleleft', \, \triangleright', \, f')$ be two group
extending structures associated to $(h_0, D)$ and $(h'_0, D') \in
{\mathcal T} (H)$. Then, $(\ast, \triangleleft, \, \triangleright,
\, f) \sim (\ast', \triangleleft', \, \triangleright', \, f')$ in
the sense of \deref{equiv1} if there exists a pair $(r, v)$ of
unitary maps $r: S \rightarrow H$, $v: S \rightarrow S$ such that
$v$ is a bijection on the set $S$ and the compatibility conditions
\equref{p1} - \equref{p4} are fulfilled. Since $|S| = 2$ and $v$
is a unitary bijection we obtain that $v$ is the identity map on
$S$. On the other hand, the unitary map $r: S \to H$ is given by
an element $g \in H$ such that $r (c) = g$. Taking into account
the construction of the group extending structures from
\equref{ind2} we can easily show that the compatibility conditions
\equref{p1} - \equref{p4} take the equivalent form
\equref{ind2bb}.
\end{proof}

\begin{examples}\exlabel{faraout}
1. Let $H$ be a group which has no outer automorphisms. The
typical example is $S_n$, for $n \neq 6$. Then, in this case
${\mathcal T} (H)$ identifies with the set of all pairs $(h_0, a)
\in H \times H$ such that :
\begin{equation}\eqlabel{ind2aabb}
a h_0 = h_0 a, \qquad a^2 h a^{-2} = h_0 h h_0^{-1}
\end{equation}
for all $h\in H$. Moreover, ${\mathcal K}^{2}_{\ltimes} (H, (S,
1_S)) \simeq {\mathcal T} (H)/ \sim $, where $(h_0, a) \sim (h'_0,
a')$ if and only if there exists $g\in H$ such that $h_0 = g a' g
a'^{-1} h'_0$ and $a h a^{-1} = g a' h a'^{-1} g^{-1}$, for all
$h\in H$.

2. Let $H$ be an abelian group. Then ${\mathcal T} (H)$ is the set
of all pairs $(h_0, \, D) \in H \times {\rm Aut} (H)$, such that
$D^2 = {\rm Id}_H$ and $D (h_0) = h_0$. Two such pairs  $(h_0, D)$
and  $(h'_0, D')$ are equivalent if and only if $D = D'$ and there
exists $g\in H$ such that $h_0 = g D'(g) h'_0$.

In particular, if $H = \ZZ$ we obtain that ${\mathcal
K}^{2}_{\ltimes} (\ZZ, (S, 1_S)) \cong \{ \overline{(0, \, {\rm
Id}_{\ZZ})}, \, \overline{(0, \, - {\rm Id}_{\ZZ})}, \,
\overline{(1, \, {\rm Id}_{\ZZ})} \}$. Thus, up to an isomorphism
of groups that stabilizes $\ZZ$ there are three groups that
contain $\ZZ$ as a subgroup of index $2$: the direct product $\ZZ
\times C_2$ corresponding to $\overline{(0, \, {\rm Id}_{\ZZ})}$,
the semidirect product $\ZZ \ltimes C_2$ corresponding to
$\overline{(0, \, - {\rm Id}_{\ZZ})}$ and the crossed product $\ZZ
\#^f C_2$ associated to the non-trivial cocycle coressponding to
$\overline{(1, \, {\rm Id}_{\ZZ})}$.
\end{examples}

\section{Conclusions and outlooks}
The crossed product and the bicrossed product for groups served as
models for similar constructions in other fields like: algebras,
Lie groups and Lie algebras, locally compact groups, Hopf
algebras, locally compact quantum groups, groupoids, etc. Thus,
the construction of the unified product presented in this paper at
the level of groups can be generalized to all the fields above.
The second direction for further study is related to the existence
of hidden symmetries of an $H$-principal bundle. Let $H$ and $G$
be Lie groups. To be consistent with the conventions of our
construction we consider the action of $H$ in an $H$-principal
bundle to be a left action. The $H$-principal bundle $(E, S, \pi)$
\cite{KN} is said to be (right) $G$-equivariant if there are
differentiable right actions of $G$ on $E$ and $S$ such that $\pi$
is $G$-equivariant and the left action of $H$ and the right action
of $G$ on $E$ commute; we say that the $G$-action is fiber
transitive if the action of $G$ on $S$ is transitive.

\textit{Let  $K$ be a Lie group and let $(E, S, \pi)$  be an
$K$-equivariant $H$-principal bundle. Does there exist a Lie group
$G\supset K$ endowed with actions on $E$ and $S$ that extend the
corresponding $K$-actions such that $(E, S, \pi)$ becomes
$G$-equivariant  fiber transitive?}

If such a $G$ exists we say that the bundle $(E, S, \pi)$ has
\textit{hidden symmetries}. The unified product associated to an
extending structure provides an answer to the above question for
$E = H \times S$ (the trivial $H$-bundle over $S$), $K=H$, and the
actions on $S$ and $E$ given respectively by $h\cdot
k:=s\triangleleft k$ and $(h, s)\cdot k:= (h(s\triangleright k),
s\triangleleft k)$, for all $h\in H$, $s\in S$, $k\in K$. If the
right action of $H$ on $S$ has a fixed point $1_{S}$, and $S$ has
an $H$-group structure then the bundle becomes $(H
\ltimes_{\Omega} S)$-equivariant fiber transitive. Therefore, the
existence of $H$-group structures on $S$ is related to the
existence of hidden symmetries. In the above discussion we
suppressed any reference to the differential or algebraic
structures on the $H$-bundle. The fact that our product
construction is compatible with all the additional structures is
the subject of a forthcoming study.

\section*{Acknowledgement}
We wish to thank Bogdan Ion for pointing out to us the relevance
of the constructions in this paper to the problem concerning the
existence of hidden symmetries for an $H$-principal bundle.

\end{document}